\documentclass[12pt]{amsart}
\usepackage{amsthm}
\usepackage{lscape}
\usepackage{amssymb}
\usepackage{cite}
\usepackage{geometry}
\usepackage{enumerate}
\usepackage{setspace}
\usepackage{tikz}
\usepackage{tikz-cd}
\usetikzlibrary{matrix, arrows}
\usepackage[unicode=true]
 {hyperref}
\hypersetup{
colorlinks=true,
urlcolor=black,
citecolor=blue,
linkcolor=blue,
}

\allowdisplaybreaks

\newtheorem{thm}{Theorem}[section]
\newtheorem{prop}[thm]{Proposition}

\newtheorem{cor}[thm]{Corollary}
\newtheorem{conjecture}[thm]{Conjecture}

\theoremstyle{definition}
\newtheorem{definition}[thm]{Definition}

\theoremstyle{remark}
\newtheorem{remark}[thm]{Remark}

\numberwithin{equation}{section}

\newcommand{\Perm}{\mathrm{Perm}}
\newcommand{\Aut}{\mathrm{Aut}}
\newcommand{\Hol}{\mathrm{Hol}}

\newcommand{\Inn}{\mathrm{Inn}}
\newcommand{\conj}{\mathrm{conj}}
\newcommand{\pmmod}{\hspace{-3mm}\pmod}
\newcommand{\PSL}{\mathrm{PSL}}

\begin{document}

\large 

\title[Non-abelian simple groups as the type of a Hopf-Galois structure]{Non-abelian simple groups which occur as the type of a Hopf-Galois structure on a solvable extension}
\author{Cindy (Sin Yi) Tsang}
\address{Department of Mathematics, Ochanomizu University, 2-1-1 Otsuka, Bunkyo-ku, Tokyo, Japan}
\email{tsang.sin.yi@ocha.ac.jp}
\urladdr{http://sites.google.com/site/cindysinyitsang/} 

\date{\today}

\maketitle


\begin{abstract}We determine the finite non-abelian simple groups which occur as the type of a Hopf-Galois structure on a solvable extension. In the language of skew braces, our result gives a complete list of finite non-abelian simple groups which occur as the additive group of a skew brace with solvable multiplicative group.
\end{abstract}


\tableofcontents


\section{Introduction}

Given a group $\Gamma$, let $\Perm(\Gamma)$ denote the group of all permutations on $\Gamma$. Recall that a subgroup $\Delta$ of $\Perm(\Gamma)$ is said to be \emph{regular} if its action on $\Gamma$ is both transitive and free, or equivalently, if the map
\[ \xi_\Delta : \Delta \longrightarrow \Gamma;\,\ \xi_\Delta(\delta) = \delta(1_\Gamma)\]
is bijective. Let us also write
\[ \begin{cases}
\lambda : \Gamma \longrightarrow \Perm(\Gamma); \,\ \lambda(\gamma) = (x\mapsto \gamma x),\\
\rho: \Gamma\longrightarrow \Perm(\Gamma);\,\ \rho(\gamma) = (x\mapsto x\gamma^{-1}),
\end{cases}\]
respectively, for the left and right regular representations of $\Gamma$. Clearly both $\lambda(\Gamma)$ and $\rho(\Gamma)$ are regular subgroups. The \emph{holomorph} of $\Gamma$ is defined to be
\begin{equation}\label{Hol}
 \Hol(\Gamma) = \rho(\Gamma)\rtimes \Aut(\Gamma) = \lambda(\Gamma)\rtimes \Aut(\Gamma),
\end{equation}
which is equal to the normalizer of both $\lambda(\Gamma)$ and $\rho(\Gamma)$ in $\Perm(\Gamma)$. 

\vspace{1.5mm}

Regular subgroups lying inside the holomorph are closely related to Hopf-Galois structures (on Galois extensions) and skew braces. Let us first briefly review the connections and some important applications.

\vspace{1.5mm}

Let $L/K$ be a finite Galois extension with Galois group $G$. By \cite{GP}, there is a bijection between Hopf-Galois structures on $L/K$ and regular subgroups of $\Perm(G)$ normalized by $\lambda(G)$. Explicitly, for such a regular subgroup $\mathcal{N}$, we associate to it the Hopf-Galois structure
\[ \mathcal{H} = \left\{ \sum_{\eta\in\mathcal{N}} \ell_\eta\eta \in L[\mathcal{N}]\,  \Big| \, \sum_{\eta\in\mathcal{N}} \ell_\eta\eta =\sum_{\eta\in\mathcal{N}} \sigma(\ell_\eta) \lambda(\sigma)\eta\lambda(\sigma)^{-1} \mbox{ for all }\sigma\in G \right\}\]
whose action on $L$ is defined by
\[ \left(\sum_{\eta\in\mathcal{N}} \ell_\eta\eta \right)\cdot x = \sum_{\eta\in \mathcal{N}} \ell_\eta (\eta^{-1}(1_G))(x)\mbox{ for all }x\in L, \]
and we refer to (the isomorphism class of) $\mathcal{N}$ as the \emph{type} of $\mathcal{H}$. For example, when $\mathcal{N}$ is taken to be $\rho(G)$, since it commutes with $\lambda(G)$ element-wise, we recover the so-called \emph{classical} Hopf-Galois structure $K[G]$ of type $G$. By \cite{By96}, given any group $N$ of the same order as $G$, there is in turn a  (not necessarily one-to-one) correspondence between Hopf-Galois structures on $L/K$ of type $N$ and regular subgroups of $\Hol(N)$ which are isomorphic to $G$.

\vspace{1.5mm}

Hopf-Galois structures are useful in the study of Galois modules. Let $L/K$ be a finite Galois extension of number fields or $p$-adic fields. In the classical setting, one views the ring of integers or valuation ring $\mathcal{O}_L$ of $L$ as a module over the \emph{associated order} of $L/K$, defined by
\[ \mathfrak{A}_{K[G]} = \{\alpha\in K[G]\mid \alpha\mathcal{O}_L\subseteq \mathcal{O}_L\},\]
where $G$ is the Galois group of $L/K$. But one can similarly regard $\mathcal{O}_L$ as a 

\noindent  module over the \emph{associated order} of $\mathcal{H}$, defined by
\[ \mathfrak{A}_{\mathcal{H}} = \{\alpha\in \mathcal{H}\mid \alpha\mathcal{O}_L\subseteq \mathcal{O}_L\},\]
for any Hopf-Galois structure $\mathcal{H}$ on $L/K$. In \cite{By97}, Byott exhibited extensions $L/K$ of $p$-adic fields for which $\mathcal{O}_L$ is not free over $\mathfrak{A}_{K[G]}$, but is free over $\mathfrak{A}_{\mathcal{H}}$ for some non-classical Hopf-Galois structure $\mathcal{H}$ on $L/K$. This suggests that we can gain a better understanding of the Galois module structure of $\mathcal{O}_L$ by considering all Hopf-Galois structures on $L/K$ other than the classical one.

\vspace{1.5mm}

We refer the reader to \cite{Childs book, Childs book new} for more details on Hopf-Galois structures.

\vspace{1.5mm}

Let $B$ be a set equipped with two group operations $+$ and $\circ$. Here, we do not require $+$ to be commutative despite the notation. We shall say that $B$ or the triplet $(B,+,\circ)$ is a \emph{(left) skew brace} if we have the brace relation
\[ x \circ (y + z) = x\circ y - x + x\circ z \mbox{ for all }x,y,z\in B.\]
The groups $(B,+)$ and $(B,\circ)$, respectively, are called the \emph{additive} and \emph{multiplicative} groups of $B$. Skew brace originated from \cite{Skew braces}. It was shown that given any group $N = (N,+)$, there is a bijection between group operations $\circ$ for which $(N,+,\circ)$ is a skew brace and regular subgroups of $\Hol(N)$. Explicitly, any such regular subgroup $\mathcal{G}$ yields a bijection 
\[ \xi_\mathcal{G} : \mathcal{G}\longrightarrow N;\,\ \xi_\mathcal{G}(\sigma) = \sigma(1_N),\]
and we may define $\circ$ via transport by setting
\[ x \circ y = \xi_\mathcal{G}(\xi_\mathcal{G}^{-1}(x)\cdot \xi_\mathcal{G}^{-1}(y))\mbox{ for all }x,y\in N.\]
Clearly $(N,\circ)$ is isomorphic to $\mathcal{G}$ and one can check that the brace relation is satisfied. We remark that isomorphism class of skew brace corresponds to conjugacy class of regular subgroup by the action of $\Aut(N)$.

\vspace{1.5mm}

Skew brace was defined in \cite{Skew braces} (also see \cite{Rump}) as a tool to study the Yang-Baxter equation. A \emph{(set-theoretical) solution} of the Yang-Baxter equation is a set $X$ equipped with a bijective map
\[ r:X\times X\longrightarrow X\times X;\,\ r(x,y) = (\sigma_x(y),\tau_y(x))\]
which satisfies the equation
\[ (r\times \mathrm{id})(\mathrm{id}\times r)(r\times \mathrm{id}) = (\mathrm{id}\times r)(r\times \mathrm{id})(\mathrm{id}\times r).\]
It is called \emph{non-degenerate} if the maps $\sigma_x$ and $\tau_x$ are bijective for all $x\in X$. As shown in \cite[Theorem 3.1]{Skew braces}, from any skew brace $B = (B,+,\circ)$ one can construct a non-degenerate solution 
\[ r : B\times B \longrightarrow B\times B;\,\ r(x,y) = (\gamma_x(y), \gamma^{-1}_{\gamma_x(y)}(-(x\circ y) + x + (x\circ y))),\]
where we define $\gamma_x(y) = -x + x\circ y$ for any $x,y\in B$. Some sort of converse of this is also true; see \cite[Theorem 3.9]{Skew braces}. This is the motivation behind the study of skew braces.

%
%

\vspace{1.5mm}

We now return to the discussion regarding regular subgroups of the holomorph. Given two finite groups $G$ and $N$ of the same order, let us say that the pair $(G,N)$ is \emph{realizable} if there exists a regular subgroup of $\Hol(N)$ that is isomorphic to $G$. By the above, realizability of $(G,N)$ is equivalent to the existence of the following:
\begin{enumerate}[(a)]
\item a Hopf-Galois structure of type $N$ on any Galois $G$-extension;
\item a skew brace having additive group isomorphic to $N$ and multiplicative group isomorphic to $G$.
\end{enumerate}
It is natural to ask how $G$ and $N$ must be related, in terms of their group-theoretic properties say, in order for $(G,N)$ to be realizable. For example:
\begin{enumerate}[(1)]
\item If $G$ is cyclic, then $N$ must be supersolvable \cite{Tsang solvable}. (also see \cite{Tsang cyclic})
\item If $G$ is abelian, then $N$ must be metabelian \cite{Byott soluble}. (also see \cite{Nasy,Tsang solvable})
\item If $G$ is nilpotent, then $N$ must be solvable \cite{Tsang solvable}.
\end{enumerate}
We also have the following conjecture due to Byott.

\begin{conjecture}\label{conj}Let $G$ and $N$ be finite groups of the same order for which the pair $(G,N)$ is realizable. If $G$ is insolvable, then $N$ is also insolvable.
\end{conjecture}

Conjecture \ref{conj} is known to be true in some special cases; see \cite{Byott simple,Tsang QS,Tsang ASG,Tsang solvable} for examples. Let us also remark that Byott has made some significant progress regarding this conjecture in the preprint arXiv:2205.13464.

\vspace{1.5mm}

However, the converse of Conjecture \ref{conj} is false. In other words, there are 

\noindent realizable pairs $(G,N)$ for which $N$ is insolvable while $G$ is solvable (see \cite[Corollary 1.1]{Byott soluble} for examples). The purpose of this paper is to investigate the possibilities of $N$ in such a realizable pair $(G,N)$. Focusing on almost simple groups $N$, we shall prove:

\begin{thm}\label{thm1}Let $N$ be a finite almost simple group. If there is a solvable group $G$ such that $(G,N)$ is realizable, then the socle of $N$ is isomorphic to one of the following:
\begin{enumerate}[$(a)$]
\item $\PSL_3(3), \, \PSL_3(4),\, \PSL_3(8), \, \mathrm{PSU}_3(8),\, \mathrm{PSU}_4(2),\, \mathrm{M}_{11}$;
\item $\PSL_2(q)$ with $q\neq 2,3$ a prime power.
\end{enumerate}
\end{thm}

We remark that our proof of Theorem \ref{thm1} uses a result which relies on the Classification of Finite Simple Groups (CFSG).

\vspace{1.5mm}

Restricting to non-abelian simple groups $N$, we are able to prove the converse and obtain a complete classification, as follows.

\begin{thm}\label{thm2}Let $N$ be a finite non-abelian simple group. Then, there is a solvable group $G$ such that  $(G,N)$ is realizable if and only if $N$ is isomorphic to one of the following:
\begin{enumerate}[$(a)$]
\item $\PSL_3(3), \, \PSL_3(4),\, \PSL_3(8), \, \mathrm{PSU}_3(8),\, \mathrm{PSU}_4(2),\, \mathrm{M}_{11}$;
\item $\PSL_2(q)$ with $q\neq 2,3$ a prime power.
\end{enumerate}
\end{thm}

We have hence obtained a complete classification of the finite non-abelian simple groups $N$ which can occur as
\begin{enumerate}[(a)]
\item the type of a Hopf-Galois structure on a solvable extension;
\item the additive group of a skew brace with solvable multiplicative group.
\end{enumerate}
We remark that when $G$ is fixed to be a finite non-abelian simple group, it is known by \cite{Byott simple} that $(G,N)$ is realizable if and only if $N$ is isomorphic to $G$. In fact, other than the obvious examples $\rho(G)$ and $\lambda(G)$, there are no other regular subgroups of $\Hol(G)$. In the language of Hopf-Galois structures, this means that
\begin{enumerate}[(a)]
\item the only Hopf-Galois structures on a Galois $G$-extension are the classical and \emph{canonical non-classical} ones (in the sense of \cite{Truman}).
\end{enumerate}
In the language of skew braces, this means that
\begin{enumerate}[(a)]\setcounter{enumi}{+1}
\item the only group operations $\circ$ on $G = (G,+)$ for which $(G,+,\circ)$ is a skew brace are the \emph{trivial} and \emph{almost trivial} ones, given by $x\circ y = x+y$ and $x\circ y= y+x$, respectively.
\end{enumerate}
The same is true when $G$ is fixed to be a finite quasisimple group \cite{Tsang QS}.

\vspace{1.5mm}

Note that when $N$ is a finite simple abelian group, namely a cyclic group of prime order, trivially $(G,N)$ is realizable if and only if $G$ is isomorphic to $N$. Similarly when $G$ is fixed to be a finite simple abelian group.

\vspace{1.5mm}

Regarding the proofs of Theorems \ref{thm1} and \ref{thm2}, it turns out that factorizations of groups play an important role; see Proposition \ref{basic prop}. We remark that the proof of the aforementioned facts (1),(2), and (3) given in \cite[Theorem 1.3]{Tsang solvable} also exploits the theory of factorizations of groups, though in a way different from Proposition \ref{basic prop}. 
Recall that a group $\Gamma$ is said to be \emph{factorized} by the subgroups $A$ and $B$ if $\Gamma = AB$ holds, and any such factorization is said to be \emph{exact} if $A\cap B=1$ in addition. 

\section{Preliminaries}

Let $G$ and $N$ be two finite groups of the same order. We write $\Inn(N)$ for the inner automorphism group of $N$ and let
\[ \conj : N\longrightarrow \Inn(N);\,\ \conj(\eta) = (x\mapsto \eta x \eta^{-1})\]
denote the natural homomorphism. We have the following proposition.

\begin{prop}\label{basic prop}The regular subgroups of $\Hol(N)$ isomorphic to $G$ are exactly the subsets of the form
\[ \{\rho(g(\sigma)) \cdot f(\sigma): \sigma\in G\},\]
where $f:G\longrightarrow \Aut(N)$ is a homomorphism and $g: G\longrightarrow N$ is a bijection such that $g(\sigma\tau) = g(\sigma)\cdot f(\sigma)(g(\tau))$ holds for all $\sigma,\tau\in G$. Moreover, in this case, the map $h : G\longrightarrow \Aut(N)$ defined by
\begin{equation}\label{h def} h(\sigma) = \conj(g(\sigma))f(\sigma)\end{equation}
is a homomorphism, the product $f(G)h(G)$ is a subgroup of $\Aut(N)$ containing $\Inn(N)$, and we have $f(G)\Inn(N) = h(G)\Inn(N)$.
\end{prop}

Let us note that $f$ and $h$, respectively, correspond to the projections onto $\Aut(N)$ along $\rho(N)$ and $\lambda(N)$, as given by the semidirect product decompositions of $\Hol(N)$ in (\ref{Hol}).

\begin{proof}See \cite[Proposition 2.1]{Tsang HG} for the first claim, and that $h$ is a homomorphism was verified in \cite[Proposition 3.4]{Tsang NYJM}. By (\ref{h def}) and the bijectivity of $g$, clearly $f(G)h(G)$ contains $\Inn(N)$ and $f(G)\Inn(N) = h(G)\Inn(N)$. 

\vspace{1.5mm}

To show that $f(G)h(G)$ is a subgroup of $\Aut(N)$, it suffices to check that
\[f(G)h(G) = h(G)f(G).\]
Let $\sigma,\tau\in G$ be arbitrary. Since $g$ is bijective, we may write
\begin{align*}
h(\sigma\tau)^{-1}(g(\sigma)^{-1}) &= g(\mu)\\
f(\sigma\tau)^{-1}(g(\sigma)) &= g(\nu)^{-1}
\end{align*}
for some $\mu,\nu\in G$. Then, using the relation (\ref{h def}), we obtain
\begin{align*}
f(\sigma)h(\tau) & = \conj(g(\sigma)^{-1})h(\sigma\tau)\\
& = h(\sigma\tau) \conj(h(\sigma\tau)^{-1}(g(\sigma)^{-1}))\\
& = h(\sigma\tau) \conj(g(\mu))\\
& = h(\sigma\tau\mu)f(\mu^{-1}),
\end{align*}
and similarly
\begin{align*}
h(\sigma)f(\tau) & = \conj(g(\sigma)) f(\sigma\tau)\\
& = f(\sigma\tau) \conj(f(\sigma\tau)^{-1}(g(\sigma)))\\
& = f(\sigma\tau) \conj(g(\nu)^{-1})\\
& = f(\sigma\tau\nu)h(\nu^{-1}).
\end{align*}
Thus, the equality $f(G)h(G) = h(G)f(G)$ indeed holds.
\end{proof}

The assertions concerning $f(G)$ and $h(G)$ in Proposition \ref{basic prop} are why realizability of $(G,N)$ is related to factorizations of groups. In particular, as an immediate consequence of Proposition \ref{basic prop}, we have:

\begin{cor}\label{cor}If $(G,N)$ is realizable, then there is a subgroup $P$ of $\Aut(N)$ containing $\Inn(N)$ such that $P= AB$ is factorized by some subgroups $A$ and $B$ which are quotients of $G$ and satisfy $A\Inn(N) = B\Inn(N)$.
\end{cor}

Corollary \ref{cor} is enough to prove Theorem \ref{thm1}. But to prove the backward implication of Theorem \ref{thm2}, we need some sort of converse to Corollary \ref{cor}. To that end, we shall use the following definition from \cite{BC}.

\begin{definition}For any group $\Gamma$, a pair $\varphi,\psi : G\longrightarrow \Gamma$ of homomorphisms is said to be \emph{fixed point free} if $\varphi(\sigma) = \psi(\sigma)$ holds only when $\sigma = 1_G$.
\end{definition}

As shown in \cite[Section 2]{BC}, given any fixed point free pair $\varphi,\psi : G\longrightarrow N$ of homomorphisms, the subset
\begin{equation}\label{fpf subgroup}
 \{ \lambda(\varphi(\sigma)) \rho(\psi(\sigma)) : \sigma\in G \}
 \end{equation}
of $\Hol(N)$ is a regular subgroup isomorphic to $G$, which implies that $(G,N)$ is realizable. As noted in \cite[Remark 7.2]{Byott soluble}, one can use this to show that:

\begin{prop}\label{exact factor}If $N = AB$ is exactly factorized by subgroups $A$ and $B$, then the pair $(A\times B,N)$ is realizable.
\end{prop}
\begin{proof}Note that $A\times B$ does have the same order as $N$. The claim follows from the above discussion because $\varphi,\psi : A\times B \longrightarrow N$ given by $\varphi(a,b) = a$ and $\psi(a,b) = b$
is clearly a fixed point free pair of homomorphisms.
\end{proof}

In proving the backward implication of Theorem \ref{thm2}, if $N=AB$ has an exact factorization by solvable subgroups $A$ and $B$, then we can just apply Proposition \ref{exact factor}. But $\PSL_3(4)$ and $\mathrm{PSU}_4(2)$, for example, do not have such a factorization and we would need another way to deal with these groups.

\vspace{1.5mm}

For simplicity, let us focus on groups $N$ having trivial center. In this case, the map $\conj : N\longrightarrow \Inn(N)$ is an isomorphism, and one sees that the pair $f,h: G\longrightarrow \Aut(N)$ of homomorphisms in Proposition \ref{basic prop} is fixed point free because the map $g$ there is bijective and sends $1_G$ to $1_N$. Moreover, one can recover $g$ from the pair $f,h$ using (\ref{h def}). We may then rephrase Proposition \ref{basic prop} as follows, where (\ref{fpf subgroup'}) is a generalization of (\ref{fpf subgroup}); see Remark \ref{remark1}.

\begin{prop}\label{centerless}Assume that $N$ has trivial center. The regular subgroups of $\Hol(N)$ isomorphic to $G$ are exactly the subsets of the form
\begin{equation}\label{fpf subgroup'} \{ \rho(\conj^{-1}(h(\sigma)f(\sigma)^{-1}) ) \cdot f(\sigma) : \sigma\in G \},\end{equation}
where $f,h : G\longrightarrow \Aut(N)$ is a fixed point free pair of homomorphisms such that $f(\sigma)\equiv h(\sigma)\hspace{-1mm}\pmod{\Inn(N)}$ holds for all $\sigma\in G$.
\end{prop}
\begin{proof}It is clear from Proposition \ref{basic prop} and (\ref{h def}) that every regular subgroup of $\Hol(N)$ isomorphic to $G$ is of the stated shape. The converse is also true  by the calculation in \cite[Proposition 3.4]{Tsang NYJM}.\end{proof}

\begin{remark}\label{remark1}In the case that $N$ has trivial center, taking $f,h$ to have image lying inside $\Inn(N)$, we may recover the construction (\ref{fpf subgroup}) from (\ref{fpf subgroup'}). This is because then we can write
\[ f(\sigma) = \conj(\varphi(\sigma)) \mbox{ and }h(\sigma) = \conj(\psi(\sigma))\]
for homomorphisms $\varphi,\psi : G\longrightarrow N$. For any $\sigma\in G$, we see that
\begin{align*}
\rho(\conj^{-1}(h(\sigma)f(\sigma)^{-1}))\cdot f(\sigma)
& =\rho(\psi(\sigma)\varphi(\sigma)^{-1})\cdot \conj(\varphi(\sigma))\\
& = \rho(\psi(\sigma)\varphi(\sigma)^{-1})\cdot \rho(\varphi(\sigma))\lambda(\varphi(\sigma))\\
& = \lambda(\varphi(\sigma)) \rho(\psi(\sigma)).
\end{align*}
This shows that (\ref{fpf subgroup'}) and (\ref{fpf subgroup}) yield the same subgroup.
\end{remark}

We may now generalize Proposition \ref{exact factor} as follows; see Remark \ref{remark2}.

\begin{prop}\label{exact factor'} Assume that $N$ has trivial center and let $P$ be a subgroup of $\Aut(N)$ containing $\Inn(N)$. If $P = AB$ is exactly factorized by subgroups $A$ and $B$ such that $A\Inn(N) = B\Inn(N)$ and $A$ splits over $A\cap \Inn(N)$, then the pair $(A\cap\Inn(N) \rtimes_\alpha B, N)$ is realizable for a suitable choice of $\alpha$.
\end{prop}
\begin{proof}Write $A_0=A\cap \Inn(N)$ and $B_0 = B\cap \Inn(N)$. Under the hypothesis, there is a subgroup $C$ of $A$ such that $A=A_0\rtimes C$, and $B/B_0\simeq C$ via
\[\begin{tikzcd}[column sep = 1cm]
\theta:\dfrac{B}{B_0} \arrow{r}{\simeq}& \dfrac{B\Inn(N)}{\Inn(N)} \arrow[equal]{r}& \dfrac{A\Inn(N)}{\Inn(N)} \arrow{r}{\simeq}& \dfrac{A}{A_0}\arrow{r}{\simeq} & C,
\end{tikzcd}\]
where all of the appearing isomorphisms are the natural ones. Since $C$ acts on $A_0$ via conjugation, this induces an action of $B$ on $A_0$. Explicitly, define
\[ \alpha : B\longrightarrow\Aut(A_0);\,\ \alpha(b) = (x\mapsto \theta(bB_0)x\theta(bB_0)^{-1}).\]
Notice that $AB = A\Inn(N)$, where $\supseteq$ holds trivially since $P=AB$ contains $\Inn(N)$, while $\subseteq$ holds because $A\Inn(N) = B\Inn(N)$. We then deduce that
\[ |\Inn(N)| = \frac{|AB|}{[A\Inn(N) : \Inn(N)]} = \frac{|A||B|/ |A\cap B|}{[A:A_0]} = |A_0||B|.\]
where $A\cap B=1$ because the factorization $P= AB$ is assumed to be exact. It follows that  $A_0\rtimes_\alpha B$ indeed has the same order as $\Inn(N)\simeq N$.

\vspace{1.5mm}

Now, let us define $f,h : A_0\rtimes_\alpha B \longrightarrow \Aut(N)$ by setting
\[ f(a,b) = a\theta(bB_0) \mbox{ and } h(a,b) = b\mbox{ for }a\in A_0, \, b\in B.\]
Clearly $h$ is a homomorphism, and $f$ is also a homomorphism because 
\begin{align*}
f((a_1,b_1)(a_2,b_2)) & =  f(a_1 \alpha(b_1)(a_2),b_1b_2)\\
& = a_1 \theta(b_1B_0)a_2\theta(b_1B_0)^{-1}\cdot \theta(b_1b_2B_0)\\
& = a_1\theta(b_1B_0) \cdot a_2\theta(b_2B_0)\\
& = f(a_1,b_1)f(a_2,b_2)
\end{align*}
for all $a_1,a_2\in A_0$ and $b_1,b_2\in B$. Note that the images of $f$ and $h$ are equal to $A$ and $B$, respectively.  Since $A\cap B =1$, we have $f(G)\cap h(G) = 1$.

\vspace{1.5mm}

Let $a\in A_0$ and $b\in B$ be arbitrary. We see that
\[ f(a,b) = h(a,b) \implies a\theta(bB_0) = 1,\,  b = 1 \implies a = 1,\, b =1,\]
which means that $(f,h)$ is fixed point free. Moreover, since $AB = A\Inn(N)$ as noted above and $A=A_0\rtimes C$ with $A_0\subseteq \Inn(N)$, we may write 
\[b = c \cdot \conj(\eta)\mbox{ for some $c\in C$ and $\eta\in N$}.\]
By the definition of $\theta$, we have $\theta(bB_0) = c$, so then
\[ f(a,b) \equiv ac \equiv c \equiv b \equiv h(a,b)\pmmod{\Inn(N)}\]
because $a\in \Inn(N)$. The claim now follows from Proposition \ref{centerless}.
\end{proof}

\begin{remark}\label{remark2}In the case that $N$ has trivial center, taking $P = \Inn(N)\simeq N$, we may recover Proposition \ref{exact factor} from Proposition \ref{exact factor'}. Indeed, for any exact factorization $\Inn(N) = AB$, it is trivial that $A\Inn(N) = B\Inn(N)$ and that $A = A\cap \Inn(N)$ splits over itself.  Thus, in the proof of Proposition \ref{exact factor'}, we have $C =1$ and so $\alpha$ is simply the trivial homomorphism. This implies that $A\rtimes_\alpha B = A\times B$, and the $f,h$ constructed there are 
the natural projections onto $A,B$, respectively, just like the $\varphi,\psi$ in the proof of Proposition \ref{exact factor}. 
\end{remark}

We remark that part of our proof of Theorem \ref{thm2} involves computations in \textsc{Magma} \cite{magma}, and in our codes, we only compute with representatives $P$ of the conjugacy classes of subgroups in $\Aut(N)$. For each $P$, similarly we only compute with representatives $A,B$ of the conjugacy classes of subgroups in $P$. By the two propositions below, as far as the conditions in Corollary \ref{cor} and Proposition \ref{exact factor'} are concerned, no generality is lost in doing so.

 \vspace{1.5mm}
 
For any $\Gamma\subseteq\Aut(N)$ and $\pi\in \Aut(N)$, let us write $\Gamma_\pi = \pi\Gamma\pi^{-1}$.

\begin{prop}\label{conjP}Let $P$ be a subgroup of $\Aut(N)$ and let $A,B$ be subgroups of $P$. For any $\pi\in \Aut(N)$, the following hold.
\begin{enumerate}[(a)]
\item If $P$ contains $\Inn(N)$, then $P_\pi$ also contains $\Inn(N)$.
\item If $P=AB$, then $P_\pi = A_\pi B_\pi$.
\item If $A\cap B =1$, then $A_\pi\cap B_\pi=1$.
\item If $A\Inn(N) = B\Inn(N)$, then $A_\pi \Inn(N) = B_\pi\Inn(N)$.
\item If $A$ splits over $A\cap \Inn(N)$, then $A_\pi$ splits over $A_\pi\cap\Inn(N)$.
\end{enumerate}
\end{prop}
\begin{proof}Parts (a), (d), and (e) hold simply because $\Inn(N)$ is a normal subgroup of $\Aut(N)$, whereas parts (b) and (c) are obvious. 
\end{proof}

\begin{prop}\label{conjAB}Let $P=AB$ be a subgroup of $\Aut(N)$ containing $\Inn(N)$ that is factorized by subgroups $A,B$. For any $\pi_1,\pi_2\in P$, the following hold.
\begin{enumerate}[(a)]
\item We also have the factorization $P = A_{\pi_1}B_{\pi_2}$.
\item If $A\cap B = 1$, then $A_{\pi_1}\cap B_{\pi_2}=1$.
\item If $A\Inn(N) = B\Inn(N)$, then $A_{\pi_1}\Inn(N) =B_{\pi_2}\Inn(N)$.
\end{enumerate}
\end{prop}


\begin{proof}Since $P = AB = BA$, we may write
\[ \pi_1 = b_1a_1,\,\ \pi_2 = a_2b_2,\,\ b_1^{-1}a_2 = ab,\]
where $a_1,a_2,a\in A$ and $b_1,b_2,b\in B$. It follows that
\[ A_{\pi_1}B_{\pi_2} = A_{b_1}B_{a_2} = b_1 A abB a_2^{-1} = b_1ABa_2^{-1} = b_1BAa_2^{-1}=BA = P,\]
which proves (a). This in turn implies that
\[ |A\cap B| = \frac{|A||B|}{|AB|} = \frac{|A_{\pi_1}||B_{\pi_2}|}{|A_{\pi_1}B_{\pi_2}|}=|A_{\pi_1}\cap B_{\pi_2}|,\]
which yields (b). Finally, suppose that $A\Inn(N) = B\Inn(N)$. Since $\Inn(N)$ is a normal subgroup of $\Aut(N)$, we have
\begin{align*}
 A_{\pi_1}\Inn(N) 
  &= (A\Inn(N))_{b_1}
  = (B\Inn(N))_{b_1}
  = B\Inn(N)
  =A\Inn(N)\\
  B_{\pi_2}\Inn(N)
  &= (B\Inn(N))_{a_2}
  = (A\Inn(N))_{a_2}
  = A\Inn(N),
 \end{align*}
 which shows (c). This completes the proof.
\end{proof}

\section{Proof of Theorem \ref{thm1}}

Let $N$ be a finite almost simple group and assume that there is a solvable group $G$ for which $(G,N)$ is realizable. By Corollary \ref{cor}, there is a subgroup $P$ of $\Aut(N)$ containing $\Inn(N)$ such that $P=AB$ for some subgroups $A,B$ which are quotients of $G$. Here, there are two important points:
\begin{itemize}
\item Since $P$ lies between $\Inn(N)$ and $\Aut(N)$, it is also
 almost simple and its socle is isomorphic to that of $N$. Indeed, suppose that 
 \[ S \leq N \leq \Aut(S),\]
 where $S$ is a non-abelian simple group identified with $\Inn(S)$. Then as is well-known (or see \cite[Lemma 4.3]{Tsang NYJM}), the natural homomorphism 
\[\Aut(N)\longrightarrow\Aut(S);\,\ \varphi \mapsto \varphi|_S\]
induced by restriction is injective and it maps $\Inn(N)$ to $N$. It follows that $P$ is also embedded between $S$ and $\Aut(S)$, which means that $P$ is an almost simple group with socle isomorphic to $S$.
\vspace{1.5mm}
\item Since $G$ is solvable, its quotients $A$ and $B$ are also solvable. 
\end{itemize}
Almost simple groups which are factorizable as the product of two solvable subgroups have been studied in \cite[Proposition 4.1]{factor} (its proof is based on a result of \cite{Kazarin} and some computations in \textsc{Magma} \cite{magma}; it requires CFSG). In particular, it tells us that the socle of $P$, which is isomorphic to the socle of $N$, must be one of the groups stated in the theorem.

\section{Proof of Theorem \ref{thm2}}

In view of Theorem \ref{thm1}, it suffices to prove the backward implication. We shall split the finite non-abelian simple groups $N$ in question into four families, as follows. The citations in the parentheses indicate the propositions to be used to deal with the family.
\begin{enumerate}[(1)]
\item $N = \PSL_3(3), \mathrm{M}_{11}$ (Proposition \ref{exact factor});
\item $N = \PSL_3(4),\PSL_3(8), \mathrm{PSU}_4(2)$ (Proposition \ref{exact factor'});
\item $N=\mathrm{PSU}_3(8)$ (Proposition \ref{centerless});
\item $N=\PSL_2(q)$ with $q\neq 2,3$ a prime power (Proposition \ref{exact factor} for an even $q$, and Proposition \ref{exact factor'} for an odd $q$).
\end{enumerate}
 
 \begin{proof}[Proof of $(1)$] By running \textsc{Code 1} in the appendix, we find that $N=AB$ is exactly factorized by some solvable subgroups $A$ and $B$. It then follows from Proposition \ref{exact factor} that the pair $(G,N)$ is realizable for the group $G = A\times B$, which is clearly solvable. 
 \end{proof}
 
\begin{proof}[Proof of $(2)$] By running \textsc{Code 2} in the appendix, we find that $\Aut(N)$ has a subgroup $P$ containing $\Inn(N)$ such that $P=AB$ is exactly factorized by some solvable subgroups $A$ and $B$ for which 
\begin{itemize}
\item $A\Inn(N) = B\Inn(N)$;
\item $A$ splits over $A\cap \Inn(N)$.
\end{itemize}
Proposition \ref{exact factor'} then implies that the pair $(G,N)$ is realizable for a suitable semidirect product $G= (A\cap\Inn(N))\rtimes B$, which is clearly solvable.
\end{proof}

\begin{proof}[Proof of $(3)$] By Proposition \ref{centerless}, we only need to find a solvable group $G$ of the same order as $N$ and also a fixed point free pair $f,h: G\longrightarrow \Aut(N)$ of homomorphisms satisfying
\begin{equation}\label{cond}
f(\sigma) \equiv h(\sigma)\hspace{-3mm}\pmod{\Inn(N)}
\end{equation}
for all $\sigma\in G$. We shall do so using \textsc{Code 3} and \textsc{Code 4} in the appendix. 

\noindent Let us explain what we computed in these codes.
\begin{enumerate}[(i)]
\item \textsc{Code 3:} Finding candidates for $f(G)$ and $h(G)$.
\\ We find that up to conjugacy in $\Aut(N)$, there is only one subgroup $P$ of $\Aut(N)$ containing $\Inn(N)$ which admits a factorization $P=AB$ by solvable subgroups $A,B$ such that $|A|,|B|$ divide $|N|$ and the equality $A\Inn(N) = B\Inn(N)$ holds. Without loss of generality, let us take $A$ to have smaller order than $B$. We note that then
\[|A| = 513,\, |B| = 96768,\, A\cap B=1,\, [P:\Inn(N)]=9,\]
as computed in the code. The $f,h : G\longrightarrow\Aut(N)$ to be constructed will have images $f(G) = A$ and $h(G) = B$.
\vspace{1.5mm}
\item \textsc{Code 4 Part I:} Finding a candidate for $G$.
\\ Note that the subgroup $\rho(N)\rtimes A$ of $\Hol(N)$ is naturally isomorphic to the outer semidirect product $\Inn(N) \rtimes A$, where $A$ acts on $\Inn(N)$ via conjugation in $\Aut(N)$. We construct this semidirect product and find that up to conjugacy, it has only one solvable subgroup $G$ of the same order as $N$.
\vspace{1.5mm}
\item \textsc{Code 4 Part II:} Finding candidates for $\ker(f)$ and $\ker(h)$.
\\ Note that $|N| = 5515776$, and that
\[ \frac{|G|}{|A|}= \frac{5515776}{513} = 10752,\,\
 \frac{|G|}{|B|} = \frac{5515776}{96768} = 57.\]
 We find that $G$ has a unique normal subgroup $K_f$ of order $10752$, and similarly a unique normal subgroup $K_h$ of order $57$. We also have 
 \[ K_f\cap K_h=1,\, G/K_f\simeq A,\, G/K_h\simeq B,\]
 as verified in the code. The $f,h : G\longrightarrow\Aut(N)$ to be constructed will have kernels $\ker(f) =K_f$ and $\ker(h) = K_h$.
\vspace{1.5mm}
\item \textsc{Code 4 Part III:} Finding a desired fixed point free pair $(f,h)$.
\\ Let $q_f : G\longrightarrow G/K_f$ and $q_h : G\longrightarrow G/K_h$ denote the natural quotient maps. Also, we construct isomorphisms
\[ \varphi_f : G/K_f\longrightarrow A\mbox{ and }\varphi_h : G/K_h\longrightarrow B. \]
For any $\pi_A\in \Aut(A)$ and $\pi_B\in \Aut(B)$, we have the homomorphisms
\[ f, h: G\longrightarrow \Aut(N);\,\
\begin{cases}f = \pi_A \circ \varphi_f \circ q_f,\\ h = \pi_B\circ \varphi_h\circ q_h.\end{cases}\]
From their definitions, it is clear that
\[ f(G) = A,\,  \ker(f) = K_f,\, h(G) = B,\, \ker(h) = K_h.\]
Since $A\cap B = 1$ and $K_f\cap K_h=1$, we see that $(f,h)$ is fixed point free. We also need (\ref{cond}) to hold for all $\sigma\in G$, or equivalently for any set of generators $\sigma\in G$. Using the generators of $G, \, \Aut(A),\, \Aut(B)$ given by \textsc{Magma}, we find that (\ref{cond}) holds for all generators $\sigma$ of $G$ for suitable choices of $\pi_A\in \Aut(A)$ and $\pi_B\in \Aut(B)$.
\end{enumerate} 
We now conclude from Proposition \ref{centerless} that $(G,N)$ is realizable, where $G$ is solvable by construction.
\end{proof}

\begin{proof}[Proof of $(4)$]We regard $N$ as a normal subgroup of $\mathrm{PGL}_2(q)$ via the natural embedding, and similarly $\mathrm{PGL}_2(q)$ as a subgroup $\Aut(N)$ by letting it act on $N$ via conjugation. In other words, we shall view
\[ N \leq \mathrm{PGL}_2(q) \leq \Aut(N),\]
where $N$ is identified with $\Inn(N)$. Note that
\[
|\mathrm{PSL}_2(q)| = \frac{q(q-1)(q+1)}{\gcd(2,q-1)},\,\ 
|\mathrm{PGL}_2(q)|  = q(q-1)(q+1).
\]
As explained in \cite[Section 3.1]{factor}, we may factor $\mathrm{PGL}_2(q)$ as the product of a Singer cycle and the stabilizer of a one-dimensional subspace. Below, let us describe this factorization explicitly in terms of matrices.

\vspace{1.5mm}

Let $\beta$ be a generator of $\mathbb{F}_{q^2}^\times$ and let $X^2 + cX + d$ be its minimal polynomial over $\mathbb{F}_q$. For any $x_1,y_1,x_2,y_2\in\mathbb{F}_q$, from the relation $\beta^2 = -d-c\beta$, we have
\begin{align*}
(x_1 + y_1\beta )(x_2 + y_2\beta ) &= (x_1x_2 + y_1y_2\beta^2) + (x_1y_2+x_2y_1)\beta\\
&=(x_1 x_2 - dy_1y_2) + (x_1y_2 + x_2y_1 - cy_1y_2)\beta.
\end{align*}
By associating $x+y\beta$ to the matrix $\left[\begin{smallmatrix} x & - dy\\ y & x-cy\end{smallmatrix}\right]$, we see that multiplication in 

\noindent $\mathbb{F}_{q^2}^\times$ corresponds to matrix multiplication, so we obtain a cyclic subgroup
\[ \widetilde{A} = \left\{
\begin{bmatrix}
x & -dy\\
y & x-cy
\end{bmatrix}: x,y\in \mathbb{F}_q,\, (x,y) \neq (0,0)
\right\}\]
of $\mathrm{GL}_2(q)$ of order $q^2 - 1$ generated by $\left[\begin{smallmatrix} 0 & -d\\ 1 & -c\end{smallmatrix}\right]$. We also have the subgroup
\[ \widetilde{B} = \left\{ 
\begin{bmatrix}
u & v \\
0 & w 
\end{bmatrix} : u,w\in\mathbb{F}_{q}^\times,\, v\in\mathbb{F}_q
\right\}\]
of $\mathrm{GL}_2(q)$ of order $q(q-1)^2$; this is the stabilizer of the subspace generated by $\left[\begin{smallmatrix}1\\0\end{smallmatrix}\right]$, and is naturally isomorphic to $\mathbb{F}_q\rtimes (\mathbb{F}_q^\times\times \mathbb{F}_q^\times)$. Clearly $\widetilde{A}\cap \widetilde{B}$ equals the center $Z$ of $\mathrm{GL}_2(q)$. We then see that
\[ \mathrm{PGL}_2(q) = AB\mbox{ with }A\cap B = 1,\, \mbox{ where }A =  \widetilde{A}/Z\mbox{ and } B = \widetilde{B}/Z\]
are clearly solvable subgroups of order $q+1$ and $q(q-1)$, respectively.

\vspace{1.5mm}

First, suppose that $q\equiv 0\hspace{-1mm}\pmod{4}$. We then have $\mathrm{PGL}_2(q) = \mathrm{PSL}_2(q)$, and so by Proposition \ref{exact factor}, the pair $(G, N)$ is realizable for $G = A\times B$, which is clearly solvable.

\vspace{1.5mm}

Next, suppose that $q\not\equiv0\hspace{-1mm}\pmod{4}$. Observe that $d$, which is the norm of $\beta$ over $\mathbb{F}_q$, generates $\mathbb{F}_q^\times$ because $\beta$ is a generator of $\mathbb{F}_{q^2}^\times$ and the norm map for finite fields is surjective. For any $u\in\mathbb{F}_q^\times$, because $\det\left[\begin{smallmatrix} 0 & -d\\ 1 & -c\end{smallmatrix}\right]=d$ we then see that there exists $\widetilde{a}\in \widetilde{A}$ for which $\det(\widetilde{a}) = u$, and clearly there exists $\widetilde{b}\in \widetilde{B}$ for which   $\det(\widetilde{b}) = u$. It follows that $\widetilde{A}\cdot \mathrm{SL}_2(q) = \widetilde{B}\cdot \mathrm{SL}_2(q)$, which implies
\[A \cdot \PSL_2(q) = B \cdot \PSL_2(q).\]
We also deduce that the homomorphisms
\[ A\longrightarrow \mathbb{F}_p^\times/(\mathbb{F}_p^\times)^2,\,\ B\longrightarrow \mathbb{F}_p^\times/(\mathbb{F}_p^\times)^2 \]
induced by the determinant map are surjective. Thus, their kernels
\[ A_0 = A\cap \PSL_2(q)\mbox{ and }B_0 = B\cap \PSL_2(q)\]
have index $2$ in $A$ and $B$, respectively. In particular, we have
\[ |A_0| = \frac{q+1}{2},\,\ |B_0| = \frac{q(q-1)}{2}.\]
By the Schur-Zassenhaus theorem, we know that $A$ splits over $A_0$ when $|A_0|$\par\noindent is odd, and that $B$ splits over $B_0$ when $|B_0|$ is odd. From Proposition \ref{exact factor'}, it then follows that $(G,N)$ is realizable for the solvable group
\[ G = \begin{cases}
A_0\rtimes_\alpha B&\mbox{when }q\equiv 1\hspace{-3mm}\pmod{4},\\
B_0\rtimes_\alpha A&\mbox{when }q\equiv 3\hspace{-3mm}\pmod{4},
\end{cases}\]
where $\alpha$ is a suitable choice of homomorphism.
\end{proof}

This completes the proof of the theorem.

\section{Appendix: \textsc{Magma} codes}

\noindent\textsc{Code 1:}
\begin{verbatim}
N:=; 
//input PSL(3,3) and M11 for N
SolSub:=[S`subgroup:S in SolvableSubgroups(N)];
exists{<A,B>:A,B in SolSub|#(A meet B) eq 1 and #N eq #A*#B};
//output: true
\end{verbatim}

\vspace{1.5mm}

\noindent\textsc{Code 2}:
\begin{verbatim}
N:=; 
//input PSL(3,4), PSL(3,8), and PSU(4,2) for N
Aut:=PermutationGroup(AutomorphismGroup(N));
Inn:=Socle(Aut);
PP:=[P`subgroup:P in Subgroups(Aut)|Inn subset P`subgroup];
L:=[];
//A list to contain the P satisfying the desired conditions.
for P in PP do
SolSub:=[S`subgroup:S in SolvableSubgroups(P)];
  if exists{<A,B>:A,B in SolSub|
      #(A meet B) eq 1 and #P eq #A*#B and
      sub<P|A,Inn> eq sub<P|B,Inn> and
      HasComplement(A,A meet Inn)}
  then Append(~L,P);
  end if;
end for;
not IsEmpty(L);
//output: true
\end{verbatim}

\vspace{1.5mm}

\noindent\textsc{Code 3:}
\begin{verbatim}
N:=PSU(3,8);
Aut:=PermutationGroup(AutomorphismGroup(N));
Inn:=Socle(Aut);
PP:=[P`subgroup:P in Subgroups(Aut)|Inn subset P`subgroup];
for p in [1..#PP] do
P:=PP[p];
SolSub:=SolvableSubgroups(P:OrderDividing:=#N);
  for a in [1..#SolSub] do
  A:=SolSub[a]`subgroup;
    for b in [a..#SolSub] do
    B:=SolSub[b]`subgroup;
      if #A*#B eq #P*#(A meet B)
      and sub<P|A,Inn> eq sub<P|B,Inn> then
      <p,a,b,#A,#B,#(A meet B),#P/#N>;
      end if;
    end for;
  end for;
end for;
//output: <8, 127, 217, 513, 96768, 1, 9>
\end{verbatim}

\vspace{1.5mm}

\noindent\textsc{Code 4:}
\begin{verbatim}
//Setting up the P, A, and B that we found in Code 3.
N:=PSU(3,8);
Aut:=PermutationGroup(AutomorphismGroup(N));
Inn:=Socle(Aut);
PP:=[P`subgroup:P in Subgroups(Aut)|Inn subset P`subgroup];
P:=PP[8];
SolSub:=SolvableSubgroups(P:OrderDividing:=#N);
A:=SolSub[127]`subgroup;
B:=SolSub[217]`subgroup;

"Part I";
AutInn:=AutomorphismGroup(Inn);
xi:=hom<A->AutInn|a:->hom<Inn->Inn|x:->a^(-1)*x*a>>;
InnxA:=SemidirectProduct(Inn,A,xi);
GG:=SolvableSubgroups(InnxA:OrderEqual:=#N);
#GG;
//output: 1
G:=GG[1]`subgroup;

"Part II";
NorSubf:=NormalSubgroups(G:OrderEqual:=10752);
NorSubh:=NormalSubgroups(G:OrderEqual:=57);
<#NorSubf,#NorSubh>;
//output: <1, 1>
Kf:=NorSubf[1]`subgroup;
Kh:=NorSubh[1]`subgroup;
<#(Kf meet Kh),IsIsomorphic(G/Kf,A),IsIsomorphic(G/Kh,B)>;
//output: <1, true, true>

"Part III";
Qf, qf:=quo<G|Kf>;
Qh, qh:=quo<G|Kh>;
isof, phif:=IsIsomorphic(Qf,A);
isoh, phih:=IsIsomorphic(Qh,B);
GenG:=Generators(G);
GenAutA:=Generators(AutomorphismGroup(A));
GenAutB:=Generators(AutomorphismGroup(B));
Out, q:=quo<Aut|Inn>;
exists{<piA,piB>:piA in GenAutA, piB in GenAutB|
  forall{g:g in GenG|(qf*phif*piA*q)(g) eq (qh*phih*piB*q)(g)}
  };
//output: true
\end{verbatim}

\vspace{1.5mm}

\section*{Acknowledgments} 

This work was supported by JSPS KAKENHI (Grant-in-Aid for Research Activity Start-up) Grant Number 21K20319. 

\vspace{1.5mm}

We thank the referee for their comments, which helped improve the exposition of the introduction significantly.

\end{document}